\documentclass[12, reqno]{amsart}
\usepackage{amsmath}
\usepackage{amscd,amsthm,amsfonts,amsopn,amssymb,verbatim}
\parskip =\medskipamount
\numberwithin{equation}{section}
\newtheorem{theorem}{Theorem} %[section]

\newtheorem{lemma}[theorem]{Lemma}

\theoremstyle{remark}
\DeclareMathOperator{\supp}{supp\,}

\def\be{\begin{equation}}
\def\ee{\end{equation}}

\def\ve{\varepsilon}

\begin{document}
\large

\title{A Quantitative Oppenheim Theorem for Generic Diagonal Quadratic Forms}
\date{\today}
\author{J.~Bourgain}
\address{J.~Bourgain, Institute for Advanced Study, Princeton, NJ 08540}
\email{bourgain@math.ias.edu}
\thanks{The author was partially supported by NSF grants DMS-1301619}
\begin{abstract}
We establish a quantitative version of Oppenheim's conjecture for one-parameter families of ternary indefinite quadratic forms using an analytic number theory approach.
The statements come with power gains and in some cases are essentially optimal.
\end{abstract}
\maketitle

\section
{Introduction}

Let $Q$ be a real nondegenerate indefinite quadratic form in $n\geq 3$ variables which is not a multiple of a form with rational coefficients.
Oppenheim's conjecture states that the set of values of $Q$ on integer vectors is a dense subset of the real line.
The conjecture was proven by Margulis [M] using methods from ergodic theory.
Thus there are functions $A(N)\to \infty$ and $\delta(N)\to 0$ with $N\to\infty$ depending on $Q$, such that
$$
\max_{|\xi|<A(N)} \ \min_{x\in\mathbb Z^n, 0<|x|<N} \ |Q(x)-\xi|< \delta (N).\eqno{(1.1)}
$$
Taking $n=3$, a quantitative version of (1.1) appears in \cite{L-M}, with $A(N)$  and $\delta(N)$ depending logarithmically on $N$.
In this Note, we consider diagonal forms of signature (2, 1)
$$
Q(x) =x_1^2 +\alpha_2x_2^2 -\alpha_3x_2^2 \qquad (\alpha_2, \alpha_3>0)\eqno{(1.2)}
$$
and prove the following for one parameter families.

\noindent
{\bf Theorem.}
{\sl Consider (1.2) with $\alpha_2>0$ fixed and taking say $\alpha_3\in [\frac 12, 1]$.
Then, for almost all $\alpha_3$, the following holds
\begin{itemize}
\item [(i)]
Assuming the Lindel\"of hypothesis for the Riemann zeta function
$$
\min_{\substack {x\in\mathbb Z^3\backslash \{0\}\\ |x|<N}} |Q(x)|\ll N^{-1+\ve} \ \text { for all } \ \ve>0.\eqno{(1.3)}
$$
Moreover (1.1) holds provided
$$
A(N)\delta (N)^{-2} \ll N^{1-\ve}\eqno{(1.4)}
$$
\item [(ii)] Unconditionally, we have
$$
\min_{\substack{x\in\mathbb Z^3\backslash \{0\}\\ |x|< N}} |Q(x)|\ll N^{-\frac 25+\ve}\eqno{(1.5)}
$$
and (1.1), assuming
$$
A(N)^3 \delta(N)^{-\frac {11}2} \ll N^{1-\ve}.\eqno{(1.6)}
$$
\end{itemize}}

Clearly, (1.3) is essentially an optimal statement.

Results on the distribution of generic quadratic forms of signature (2,~1) and (2,  2) were obtained in \cite{E-M-M} but they are not quantitative.
In [S], an analytic and quantitative approach to the pair correlation problem for generic binary quadratic forms $\alpha m^2+mn+\beta n^2$
(which amounts to the distribution of quadratic forms of (2, 2) signature) is given.
The same problem for generic diagonal forms  $m^2+an^2, \alpha>0$ is considered in \cite{B-B-R-R}, again using analytical techniques, though
different from those in [S].
The proof of the above Theorem is based on the same method (see \S8 of \cite{B-B-R-R}).
We note that this technique also enables to obtain distributional results in the sense of \cite{E-M-M} or [S], cf \cite{Bo}.

Returning to quantitative versions of the Oppenheim conjecture, there is also the recent preprint of A.~Ghosh and D.~Kelmer \cite{G-K} to be mentioned, where
the authors establish in particular (1.3) for generic members in the family of {\it all} indefinite ternary quadratic forms, which is 5-dimensional, while in our 
Theorem below a one-dimensional family is considered.
See also \S 5 of this paper.

Next, note that the Theorem is an easy consequence of the following statement.

\noindent
{\bf Proposition.}
{\sl Let $Q= Q_{\alpha_2, \alpha_3}$ be as above, $\alpha_2>0$ fixed.
Let $\xi \in\mathbb R$, \break $|\xi|< \frac 12 N^{2}$, where we have fixed $N$ sufficiently large.
\begin{itemize}
\item[(i)] Assuming Lindel\"of and taking $N^{-1+\ve}<\delta<1$, the statement
$$
\min_{\substack{ x\in\mathbb Z^3\\ 0<|x|<N}} |Q(x)-\xi|<\delta\eqno{(1.7)}
$$
holds, excluding an exceptional set in $\alpha_3\in [\frac 12, 1]$ of measure at most
$$
(\delta N^{1-\ve})^{-1}.\eqno{(1.8)}
$$
\item[(ii)] Unconditionally, the same holds with an exceptional set of measure at most
$$
\delta^{-\frac 56} N^{-\frac 13+\ve}\eqno{(1.9)}
$$
assuming $\delta>N^{-\frac 25}$.
\end{itemize}}

In order to deduce the Theorem from the Proposition, we just let $\xi$ range in a $\delta$-dense subset of $[-A, A]$.

\section
{Proof of the proposition $(i)$}

The argument is a modification of \S8 in \cite{B-B-R-R}.

Let $0\leq w_1 \leq 1, 0\leq w_2\leq 1$ be smooth bumpfunctions satisfying $w_1 =1$ on $[\frac 12, \frac 34]$, $\supp w_1 \subset [\frac 14, 1]$ and
$w_2 =1$ on $[-1, 1], \supp w_2\subset [-2, 2]$, $w_2(t)= w_2(-t)$.

We seek for a lower bound for 
$$
\sum_{x_1, x_2, x_3\in\mathbb Z} w_1\Big(\frac {x_1}N\Big) w_1\Big(\frac {x_2}N\Big) w_1\Big(\frac {x_3}N\Big) \ 1_{[|Q(x) -\xi|<\delta]}\eqno{(2.1)}
$$
or equivalently
$$
\sum_{x_1, x_2, x_3\in\mathbb Z} w_1\Big(\frac {x_1}N\Big) w_1\Big(\frac{x_2}N\Big) w_1\Big(\frac{x_3}N\Big) \ 1_{[|\log(x_1^2+\alpha_2x_2^2-\xi)-2
\log x_3 -\log \alpha_3|<\frac\delta{N^2}]}.\eqno{(2.2)}
$$
Set $T=\frac {N^2}\delta$.
Expressing (2.2) using the Fourier transform, denote
$$
\begin{aligned}
F_1(t)&= \sum_{x_1x_2\in\mathbb Z} w_1\Big(\frac {x_1}N\Big) w_1\Big(\frac {x_2}N\Big) e^{it\log (x_1^2 +\alpha_2 x_2^2-\xi)}\\
F_2(t)&= \sum_{n\in\mathbb Z} w_1\Big(\frac nN\Big) e^{it\log n}
\end{aligned}
$$
Then (2.2) amounts to
$$
\frac 1T \int_{\mathbb R} \widehat{w_2} \Big(\frac tT\Big) F_1(t) \overline{F_2(2t)} e^{-it\log\alpha_3} dt.\eqno{(2.3)}
$$
Split $\widehat{w_2} (\frac tT)$ as $\widehat{w_2}(\frac t{N^{\frac 12}}) + \big(\widehat{w_2}(\frac tT)-\widehat{w_2}(\frac t{N^{\frac 12}})\big)$
and let ($*$) and ($**$) be  the corresponding contributions to (2.3).
Clearly ($*$) amounts to
$$
\frac {N^{\frac 1{2}}} T \ \sum_{x_1, x_2, x_3 \in\mathbb Z} w_1 \Big(\frac {x_1} N\Big) w_1\Big(\frac {x_2} N\Big) w_1\Big(\frac {x_3}N\Big)\
1_{[|\log (x_1^2 +\alpha_2 x_2^2-\xi) -  2\log x_3 -\log \alpha_3|< N^{-\frac 1{2}}]}
$$
which is of the order of $\frac {N^3}T$ without further restrictions on $\alpha_3$.

Indeed, the above expression counts the number of solutions of the diophantine inequality
$$
\frac {\alpha_3x_3^2}{x_1^2+\alpha_2x_2^2-\xi} = 1+ O(N^{-\frac 12}), x_i\approx N
$$
or
$$
x_1^2+\alpha_2x_2^2 -\alpha_3x_3^2=O(N^{3/2}), x_i\approx N
$$
which has $\approx N^{5/2}$ solutions.

Hence, considering ($**$) as a function of $\alpha_3$, we need to evaluate
$$
\text{\rm mes\,} \Big[\alpha_3\in \Big[\frac 12, 1\Big] ; |(**)| \gtrsim \delta N\Big]
$$
which, by Chebyshev's inequality is bounded by $(\delta N)^{-2} \Vert (**)\Vert^2_{{L^2}_{(\alpha_3)}}$.
Since $w_2$ was assumed symmetric, $|\widehat{w_2} (\frac tT) -\widehat{w_2} (\frac t{N^{1/2}})| \leq C\min(1, \frac {t^2}N, (\frac T{|t|})^{10})$.
Hence, using Parseval
$$
\begin{aligned}
\Vert(**)\Vert^2_{L^2_{(\alpha_3)}} &\leq CT^{-2} \int\min \Big(1, \frac {t^4}{N^2}, \Big(\frac Tt\Big)^{20}\Big) |F_1(t)|^2 |F_2(t)|^2 dt\\
& < CT^{-2} N^{6-\frac 32}+CT^{-2}\int_{[|t|>N^{\frac 1{10}}]}\min\Big(1, \Big(\frac Tt\Big)^{20}\Big) |F_1|^2 |F_2|^2
\end{aligned}
$$
and
$$
(\delta N)^{-2}\Vert(**)\Vert^2_{L^2_{(\alpha_3)}}< CN^{-\frac 32} +CN^{-6} \int_{[|t|<N^{\frac 1{10}]}}\min \Big(1, \Big(\frac Tt\Big)^{20}\Big) |F_1|^2 |F_2|^2.
\eqno{(2.4)}
$$
The second term on the r.h.s. of (2.4) is further estimated by
$$
CN^{-6} \max_{|t|>N^{\frac 1{10}}} \Big(\min\Big(1, \frac T{|t|}\Big) |F_2(t)|\Big)^2. \Big[\int \min \big(1, \Big(\frac T{|t|}\Big)^{10}\Big) |F_1(t)|^2 dt\Big].
\eqno{(2.5)}
$$

From the definition of $F_1$, the last factor in (2.5) may clearly be estimated by
$$
\begin{aligned}
&T.\sum_{x_1,x_2,x_3, x_4\in\mathbb Z} w_1\Big(\frac {x_1}N\Big) w_1\Big(\frac {x_2}N\Big) w_1\Big(\frac {x_3}N\big) w_1\Big(\frac {x_4}N\Big) \,
1_{[|\log (x_1^2+\alpha_2 x_2^2-\xi)-\log (x_3^2+\alpha_2 x_4^2-\xi)|<
 \frac 1T]} \\
&\sim T\sum w_1\Big(\frac {x_1}N\Big) w_1\Big(\frac {x_2}N\Big) w_1\Big(\frac {x_3}N\Big) w_1\Big(\frac {x_4}N\Big) \
1_{[|(x_1^2-x_3^2)+\alpha_2(x_2^2 -x_4^2)|<\delta]}\\
&\ll TN^\ve \sum_{\substack{u, v\in\mathbb Z\\ |u|, |v| < N^2}} 1_{[|u+\alpha_2v|<\delta]} \ll TN^{2+\ve}=\frac 1\delta N^{4+\ve}
\end{aligned} 
$$
when the factor $N^\ve$ accounts for the multiplicity in the representations \break $u=x_1^2 -x_3^2, v=x_2^2-x_4^2$.

Next, we need to estimate $F_2(t)$. Denoting
$$
\check w_1(s) =\int^\infty_0 w_1(x) x^s \frac {dx}s
$$
the Mellin transform of $w_1$, we have
$$
F_2(t) =\int_{\text{Res\,} =2} \check w_1(s)N^s \zeta (s-it)\frac{ds}{2\pi i}
$$
where $\check w_1$ has rapid decay on vertical lines.
Shifting the line of integration to $\text {Res\,} =\frac 12$, we pick up the pole of $\zeta$ contributing to
$$
\check w_1 (1+it) N^{1+it}
$$
which for $|t|>N^{\frac 1{10}}$ is negligible due to the decay of $\check w_1$.

Hence $F_2(t)$ may be bounded by
$$
N^{\frac 12} \int^\infty_{-\infty} \frac {|\zeta(\frac 12+i(y-t))|}{1+|y|^{10}} dy \eqno{(2.6)}
$$
and which, assuming the Lindel\"of hypothesis is $\ll N^{\frac 12 }(1+|t|)^\ve$.

From the preceding, (2.5) $\ll \frac 1{\delta N^{1-\ve}}$ upon Lindel\"of, proving (1.8).

\noindent
{\bf Remark.}

Instead of using the Lindel\"of hypothesis, the bound $|\zeta (\frac 12+it)|< C(1+|t|)^{\frac 16}$ implies that $(2.5) <CN^{\frac 12}|t|^{\frac 16}$ and $(2.4)\ll
\delta^{-\frac 43}N^{-\frac 13+\ve}$.
Hence, assuming $\delta>N^{-\frac 14+\ve}$, there is an unconditional bound $\delta^{-\frac 43}N^{-\frac 13+\ve}$ on the measure of the exceptional set.
Better results will be obtained by invoking certain large values estimates on Dirichlet polynomials.

\section
{Large Values Estimates}

The following distributional inequality follows from [Ju] and we  will include a selfcontained argument here.

\begin{lemma}
Consider a Dirichlet polynomial
$$
S(t) =\sum_{n\sim N} a_n n^{it} \ \text { with } \ |a_n|\leq 1.\eqno{(3.1)}
$$
Then, for $T>N$
$$
\text{\rm mes\,} [|t|<T; |S(t)| >V]\ll N^\ve (N^2V^{-2}+N^4V^{-6} T).\eqno{(3.2)}
$$
\end{lemma}

\begin{proof}

Note first that since $\int_{|t|<T} |S(t)|^2 dt\ll N^\ve(N+T)(\sum|a_n|^2)\ll N^{1+\ve} T$, the l.h.s. of (3.2) is certainly bounded by $N^{1+\ve}TV^{-2}<
N^{4+\ve} V^{-6}T$ for $V< N^{\frac 34+\ve}$.

Hence, we may assume $V>N^{\frac 34+\ve}$.

Invoking (1.4) of the Main Theorem in [Ju], taking $G=N$, one gets for $V<N$
$$
R\ll_{\ve, k} T^\ve[N^2V^{-2} +TN^{4-\frac 2k} V^{-6+\frac 2k}+T(N^6 V^{-8})^k]\eqno{(3.3)}
$$
for any fixed positive integer $k$ and where $R$ denotes the maximal size of a 1-separated subset $\{t_r; 1\leq r\leq R\}$ of $[|t|<T; |S(t)|>V]$.

Since $V>N^{\frac 34+\ve'}$, (3.2) follows by letting $k\to\infty$.

A more direct proof is obtained as follows.

The Hal\'asz-Montgomery inequality implies that
$$
R^2V^2 \leq RN^2 +N\sum_{r\not= s} |H_N(t_r-t_s)|\eqno{(3.4)}
$$
where we take $H_N(t)=\sum_{n\sim N} n^{it}$. 
Using stationary phase, we have
$$
|H_N(t)|<c\Big(\frac N{|t|}+\sqrt t\Big)\eqno{(3.5)}
$$
so that, since the points $t_r$ are 1-separated, the last term of (3.4) may be bounded by $T^\ve N^2|R|+ c|R|^2N\sqrt T$.
Next, we break up the interval $[|t|<T]$ in intervals $I$ of size $T_0<T$, assuming that $V^2\lesssim N\sqrt T$, taking $T_0 \approx \frac {V^4}{N^2}> N$ (Huxley's
subdivision).
Since then
$$
|\{ r; t_r\in I\}|\ll N^{2+\ve} V^{-2}
$$
from the preceding and by our choice of $T_0$, the resulting bound on $R$ becomes
$$
R< N^{2+\ve}V^{-2}\Big(1+\frac T{T_0}\Big)\eqno{(3.6)}
$$
implying (3.2).
\end{proof}

\begin{lemma}
Define for $\alpha>0$
$$
S(t)=\sum_{m, n\sim N} a_{m, n}(m^2+\alpha n^2)^{it} \ \text { with } \ |a_{m, n}|\leq 1.\eqno{(3.7)}
$$
Then, for $T>N^2$
$$
\text{mes\,} [|t|<T, |S(t)|>\lambda]\ll TN^{2+\ve} \lambda^{-2}.\eqno{(3.8)}
$$
\end{lemma}

\begin{proof}
This is immediate from the mean square bound
$$
\int_{|t|<T} |S(t)|^2 dt\ll N^\ve(N^2+T) \ \Big(\sum |a_{m, n}|^2\Big).\eqno{(3.9)}
$$
\end{proof}

We also need a bound on the partial sums of the Epstein zeta function.

\begin{lemma}
For $|t|>N^2$, we have
$$
\Big|\sum_{m, n\sim N} (m^2+\alpha n^2)^{it}\Big|\ll N|t|^{\frac 13+\ve}.\eqno{(3.10)}
$$
\end{lemma}

\begin{proof}

The argument follows the steps of Van der Corput's third derivative estimate similar to the case of partial sums of the Riemann zeta function (i.e. the exponent
pair $\big(\frac 16, \frac 23)\big)$.
Details of the argument may be found in \cite{Bl}, p 5, 6.
\end{proof}

\section
{Proof of the proposition $(ii)$}

Returning to the second term in (2.4), subdivide the integral
$$
\int_{[|t|>N^{\frac 1{10}}]} = \int _{[N^{\frac 1{10}}\leq |t|\leq N^2]} + \int_{[|t|>N^{2}]} = (4.1)+(4.2).
$$
Using the bound $N^{\frac 12+\ve}|t|^{\frac 16}\ll N^{\frac 56+\ve}$ on $F_2(t)$ for $N^{\frac 1{10}} <|t|<N^2, (4.1)\ll N^{\frac {17}3+\ve}$.
Next, we evaluate (4.2).

Let $I=[N^{2}, T]$ or of the form $[T_0, T_0+T]$, $T_0\geq T$.
In view of the factor $\min \big(1, (\frac T{|t|})^{10}\big)$ it clearly suffices to consider a single interval $I$.
Introduce level sets
$$
\Omega_\lambda = [|t|\in I;  |F_1(t)|\sim\lambda]
$$
and
$$
\Omega_V' =[t\in I;  |F_2(t)|\sim V].
$$
By Lemma 2, $|\Omega_\lambda|\ll TN^{2+\ve}\lambda^{-2}$ where, by Lemma 3, we may restrict $\lambda\leq \lambda_*=NT_0^{\frac 13+}$.
Application of Lemma 1 to the Dirichlet polynomial $S(t)= F_2(t)^2=\sum_{n\sim N^2} a_n n^{it}$, $0\leq |a_n|\ll N^\ve$, obtained by shift in $t$ and 
replacing $V$ by $V^2$, implies
that $|\Omega_V|\ll N^\ve (N^4V^{-4}+N^8 V^{-12}T)$.

Hence 
$$
(4.2)<N^\ve \max_{\lambda<\lambda_*, V} (\lambda^2V^2)|\Omega_\lambda\cap\Omega_V'|)\eqno{(4.3)}
$$
where from the preceding
$$
\begin{aligned}
\lambda^2V^2 |\Omega_\lambda\cap \Omega_V'|&\ll N^\ve\min (TN^2V^2, N^4V^{-2}\lambda^2+TN^8 V^{-10}\lambda^2)\\
&\ll T^{\frac 12} N^{3+\ve}\lambda_* +TN^{3+\ve}\lambda_*^{\frac 13}\ll T_0^{\frac 56} N^{4+\ve}+ T_0^{\frac {10} 9} N^{\frac {10}3+\ve}.
\end{aligned}
$$
It follows that the l.h.s. of (2.4) may be estimated by
$$
T^{\frac 56}N^{-2+\ve}+T^{\frac {10}9} N^{-\frac 83+\ve}\ll N^{-\frac 13+\ve} \delta^{-\frac 56}+N^{-\frac 49+\ve}\delta^{-\frac {10}9}< N^{-\frac
13+\ve}\delta^{-\frac 56}
$$
\big(again in view of the factor $(\frac T{T_0})^{20}$ for $T_0\geq T$\big)
provided $\delta>N^{-\frac 25}$.

\section
{ Further comment: Generic diagonal forms}

Instead of fixing $\alpha_2$, we may consider both $\alpha_2, \alpha_3 \in [\frac 12, 1]$ as parameters,
hence the fully generic (2-parameter family) of indefinite diagonal ternary quadratic forms.
In this situation, (1.3) in the Theorem holds without the need to invoke the Lindel\"of hypothesis.

Recalling the definition of $F_1$ and $F_2$, if we have $\alpha_2$ as additional parameter at our disposal, the second term in (2.4) may be replaced by (with $\xi =0$)
$$
N^{-6} \int_{[|t|>N^{\frac 1{10}}]} \min \Big(1, \Big(\frac Tt\Big)^{20}\Big) |F_2(t)|^2 \big[Av_{\alpha_2} |F_1(t)|^2\big] dt.\eqno{(5.1)}
$$

\begin{lemma}
$$
Av_{\alpha_2} |F_1(t)|^2 \ll N^{2+\ve}+ \frac {N^{4+\ve}}{|t|}.\eqno{(5.2)}
$$
\end{lemma}

\begin{proof}
Write
$$
|F_1(t)|^2 =\sum_{x_1, x_2, x_3, x_4\sim N} e^{it [\log (x_1^2+\alpha_2x_2^2)-\log (x_3^2 \alpha_2 x_4^2)]}
$$
and note that the phase function satisfies
$$
\partial_\alpha [\log (x_1^2+\alpha x_2^2) -\log (x_3^2 +\alpha x_4^2)] \sim \frac {x_2^2x_3^2 -x_1^2x_4^2}{N^4}.
$$
Hence we may bound
$$
{Av}_{\alpha_2} |F_1(t)|^2 \leq C\sum_{x_1, x_2, x_3, x_4 \sim N}\min \Big(1, \frac {N^4}{|t|\big|x_2^2 x_3^2-x_1^2 x_4^2|}\Big).\eqno{(5.3)}
$$
Writing $|x_2^2x_3^2 -x_1^2 x^2_4|\sim N^2 |x_2x_3 -x_1 x_4|$ and distinguishing the cases $x_2x_3 -x_1x_4=0$ and
$|x_2x_3 -x_1x_4|\geq 1$, (5.2) easily follows.
\end{proof}

Since $\int_{|t|\sim 2^k}|\zeta (\frac 12+it)|^2\ll 2^{k(1+\ve)}$, we obtain from (2.6) that
$\int_{|t|\sim 2^k}|F_2(t)|^2 \ll N2^{k(1+\ve)}$.
Together with (5.2), this implies that again
$$
\begin{aligned}
(5.1)&\ll N^{-6+\ve} \sum_k \min (1, T.2^{-k})^2 N2^{k(1+\ve)} (N^2+2^{-k}N^4)\\
&\ll N^{-6+\ve} (N^3T+N^5)\ll N^{-3+\ve} T=\frac 1{\delta N^{1-\ve}}.
\end{aligned}
$$
This proves the claim.

\end{document}